\documentclass[12pt,reqno]{amsart}
\usepackage{amsmath,amssymb,amsfonts}
\usepackage{eucal}
\usepackage{graphicx,psfrag}
\usepackage{hyperref}
\usepackage{verbatim} 
\usepackage{cancel}
\usepackage{url,hyperref}

\allowdisplaybreaks

\newtheorem{lem}{Lemma}[section]
\newtheorem{cor}[lem]{Corollary}

\newtheorem{thm}[lem]{Theorem}
  \newtheorem{exama}[section]{Example}

  \newtheorem{conja}{Conjecture}[section]

\numberwithin{equation}{section}
\numberwithin{table}{section}

\renewcommand\phi{\varphi}              
\newcommand\Eta{{\mathrm H}}
\newcommand \rk{\operatorname{rank}}

\begin{document}
\title{Spectral Properties of Complex Unit Gain Graphs}

\author{Nathan Reff}
\address{Department of Mathematical Sciences\\Binghamton University (SUNY)\\ Binghamton, NY 13902-6000, U.S.A.}
\email{reff@math.binghamton.edu}



\begin{abstract}
A complex unit gain graph is a graph where each orientation of an edge is given a complex unit, which is the inverse of the complex unit assigned to the opposite orientation.  We extend some fundamental concepts from spectral graph theory to complex unit gain graphs.  We define the adjacency, incidence and Laplacian matrices, and study each of them.  The main results of the paper are eigenvalue bounds for the adjacency and Laplacian matrices.  
\end{abstract}

\date{\today}
\maketitle
\section*{Introduction}

The study of matrices and eigenvalues associated to graphs has developed over the past few decades.  Researchers have extensively studied the adjacency, Laplacian, normalized Laplacian and signless Laplacian matrices of a graph.  Recently there has been a growing study of matrices associated to a signed graph \cite{MR1950410,MR2156977,MR2477235,MR2527006,MITTOSSG,Germina20112432}.  In this paper we hope to set the foundation for a study of eigenvalues associated to a complex unit gain graph.

The \emph{circle group}, denoted $\mathbb{T}$, is the multiplicative group of all complex numbers with absolute value 1.  In other words, $\mathbb{T}=\{ z\in\mathbb{C} : |z| =1\}$ thought of as a subgroup of $\mathbb{C}^{\times}$, the multiplicative group of all nonzero complex numbers. 

A \emph{$\mathbb{T}$-gain graph} (or \emph{complex unit gain graph}) is a graph with the additional structure that each orientation of an edge is given a complex unit, called a gain, which is the inverse of the complex unit assigned to the opposite orientation.  We define $\vec{E}(\Gamma)$ to be the set of oriented edges, so this set contains two copies of each edge with opposite directions.  We write $e_{ij}$ for the oriented edge from $v_i$ to $v_j$.  Formally, a $\mathbb{T}$-gain graph is a triple $\Phi=(\Gamma,\mathbb{T},\phi)$ consisting of an \emph{underlying graph} $\Gamma=(V,E)$, the circle group $\mathbb{T}$ and a function $\phi :\vec{E}(\Gamma) \rightarrow \mathbb{T}$ (called the \emph{gain function}), such that $\phi(e_{ij})=\phi(e_{ji})^{-1}$.  For brevity, we write $\Phi=(\Gamma,\phi)$ for a $\mathbb{T}$-gain graph. 

The \emph{adjacency matrix} $A(\Phi)=(a_{ij})\in \mathbb{C}^{n\times n}$ is defined by
\begin{equation*}
a_{ij}=
\begin{cases} \phi(e_{ij}) & \text{if }v_i\text{ is adjacent to }v_j,
\\
0 &\text{otherwise.}
\end{cases}
\end{equation*}
If $v_i$ is adjacent to $v_j$, then $a_{ij}=\phi(e_{ij})=\phi(e_{ji})^{-1}=\overline{\phi(e_{ji})}=\bar{a}_{ji}$.  Therefore, $A(\Phi)$ is Hermitian and its eigenvalues are real.  The \emph{Laplacian matrix} $L(\Phi)$ (\emph{Kirchhoff matrix} or \emph{admittance matrix}) is defined as $D(\Gamma)-A(\Phi)$, where $D(\Gamma)$ is the diagonal matrix of the degrees of vertices of $\Gamma$.  Therefore, $L(\Phi)$ is also Hermitian.

We study both the adjacency and Laplacian matrices of a $\mathbb{T}$-gain graph.  The eigenvalues associated to both of these matrices are also studied.  We obtain eigenvalue bounds that depend on structural parameters of a $\mathbb{T}$-gain graph.  

A consequence of studying complex unit gain graphs is that signed and unsigned graphs can be viewed as specializations.  A \emph{signed graph} is a $\mathbb{T}$-gain graph where only $+1$ and $-1$ gains are used.  An unsigned graph can be thought of as a $\mathbb{T}$-gain graph where $+1$ is the only gain used.  Therefore, any result for the adjacency or Laplacian matrix of a $\mathbb{T}$-gain graph implies that the same result holds for signed and unsigned graphs.  Restricting the graph to have specified gains can produce other familiar matrices in the literature as well.  For instance, if we write $(\Gamma,-1)$ for the gain graph with all edges assigned a gain of $-1$.  The \emph{signless Laplacian} (\emph{quasi-Laplacian} or \emph{co-Laplacian}) matrix of an unsigned graph $\Gamma$ is $Q(\Gamma):=D(\Gamma)+A(\Gamma)=L(\Gamma,-1)$.      

\section{Background}
We will always assume that $\Gamma$ is simple.  The set of vertices is $V:=\{v_1,v_2,\ldots,v_n\}$.  Edges in $E$ are denoted by $e_{ij}=v_i v_j$.  We define $n:=|V|$ and $m:=|E|$.  The degree of a vertex $v_j$ is denoted by $d_j=\mathrm{deg}(v_j)$.  The maximum degree is denoted by $\Delta$. The set of vertices adjacent to a vertex $v$ is denoted by $N(v)$.  The \emph{average 2-degree} is $m_j:=\sum_{v_i\in N(v_j)} d_i/d_j$.  The \emph{$\mathfrak{g}$-degree} is $d_j^{\mathfrak{g}}:=|\{v_k\in N(v_j){:}\phi(e_{jk})=\mathfrak{g}\}|$.  The \emph{set of used gains} is $\mathfrak{W}:=\text{Image}(\phi)$.  All gain graphs considered in this paper will be finite, and therefore, $\mathfrak{W}$ will always be finite.   Now we can write $d_j=\sum_{\mathfrak{g}\in\mathfrak{W}} d_j^{\mathfrak{g}}$.  The \emph{net degree} is $d_j^{\text{net}}:=\sum_{\mathfrak{g}\in\mathfrak{W}} \mathfrak{g}d_j^{\mathfrak{g}}$.  We define several different types of degree vectors that will be used as follows: $\mathbf{d}:=(d_1,\ldots,d_n)$, $\mathbf{d}^{(k)}:=(d_1^k,\ldots,d_n^k)$ and $\mathbf{d}^{\text{net}}:=(d_1^{\text{net}},\ldots,d_n^{\text{net}})$.   We define $\mathbf{j}:=(1,\ldots,1)\in\mathbb{C}^n$. 

The \emph{gain of a walk} $W=e_{12}e_{23}\cdots e_{(l-1)l}$ is $\phi(W)=\phi(e_{12})\phi(e_{23})\cdots\phi(e_{(l-1)l})$.  A walk $W$ is \emph{neutral} if $\phi(W)=1$.  An edge set $S\subseteq E$ is \emph{balanced} if every cycle $C\subseteq S$ is neutral.  A subgraph is \emph{balanced} if its edge set is balanced.  We write $b(\Phi)$ for the number of connected components of $\Phi$ that are balanced. 

A \emph{switching function} is any function $\zeta:V\rightarrow \mathbb{T}$.  Switching the $\mathbb{T}$-gain graph $\Phi=(\Gamma,\phi)$ means replacing $\phi$ by $\phi^{\zeta}$, defined by: $\phi^{\zeta}(e_{ij})=\zeta(v_i)^{-1} \phi(e_{ij}) \zeta(v_j)$; producing the $\mathbb{T}$-gain graph $\Phi^{\zeta}=(\Gamma,\phi^{\zeta})$.  We say $\Phi_1$ and $\Phi_2$ are \emph{switching equivalent}, written $\Phi_1 \sim \Phi_2$, when there exists a switching function $\zeta$, such that $\Phi_2=\Phi_1^{\zeta}$.  Switching equivalence forms an equivalence relation on gain functions for a fixed underlying graph.  An equivalence class under this equivalence relation is called a \emph{switching class} of $\phi$.

A \emph{potential function} for $\phi$ is a function $\theta:V\rightarrow \mathbb{T}$, such that for every $e_{ij}\in \vec{E}(\Gamma)$, $\theta(v_i)^{-1}\theta(v_j) =\phi(e_{ij})$.
We write $(\Gamma,1)$ for the $\mathbb{T}$-gain graph with all neutral edges.

\begin{lem}[\cite{MR1007712,Math581Notes}]\label{HBGen} Let $\Phi=(\Gamma,\phi)$ be a $\mathbb{T}$-gain graph.  Then the following are equivalent:
\begin{enumerate}
\item $\Phi$ is balanced.
\item $\Phi\sim (\Gamma,1)$.
\item $\phi$ has a potential function.
\end{enumerate}
\end{lem}

The multiset of all eigenvalues of $A\in \mathbb{C}^{n\times n}$, denoted by $\sigma(A)$, is called the \emph{spectrum} of $A$.  We write $A^*$ for the conjugate transpose of the matrix $A$.  Since the eigenvalues of any Hermitian matrix $A$ are real, we assume they are labeled and ordered according to the following convention:
\[ \lambda_n(A) \leq \lambda_{n-1}(A) \leq \cdots \leq \lambda_2(A) \leq \lambda_1(A).\]

If $A\in \mathbb{C}^{n\times n}$ is Hermitian, then the quadratic form $\mathbf{x}^*A\mathbf{x}$, for some $\mathbf{x}\in\mathbb{C}^n\backslash{\{\mathbf{0}\}}$, can be used to calculate the eigenvalues of $A$. In particular, we can calculate the smallest and largest eigenvalues using the following, usually called the Rayleigh-Ritz Theorem.
\begin{lem}[\cite{MR1084815}, Theorem 4.2.2]\label{RRThm} Let $A\in\mathbb{C}^{n\times n}$ be Hermitian.  Then
\begin{align*}
\lambda_1(A)&=\max_{\mathbf{x}\in \mathbb{C}^n \backslash \{\mathbf{0}\}} \frac{\mathbf{x}^*A\mathbf{x}}{\mathbf{x}^*\mathbf{x}} = \max_{\mathbf{x}^{*}\mathbf{x}=1} \mathbf{x}^*A\mathbf{x},\\
\lambda_n(A)&=\min_{\mathbf{x}\in \mathbb{C}^n \backslash \{\mathbf{0}\}} \frac{\mathbf{x}^*A\mathbf{x}}{\mathbf{x}^*\mathbf{x}} =\min_{\mathbf{x}^{*}\mathbf{x}=1} \mathbf{x}^*A\mathbf{x}.
\end{align*}
\end{lem}

The following lemma is credited to Weyl.

\begin{lem}[\cite{MR1084815}, Theorem 4.3.1]\label{WeylINEQ} Let $A,B\in\mathbb{C}^{n\times n}$ be Hermitian.  Then for any $k\in\{1,\ldots,n\}$,
\begin{equation*}
\lambda_{k}(A)+\lambda_n(B) \leq \lambda_{k}(A+B) \leq \lambda_{k}(A)+\lambda_1(B).
\end{equation*}
\end{lem}

An $r\times r$ \emph{principal submatrix} of $A\in \mathbb{C}^{n\times n}$, denoted by $A_r$, is a matrix obtained by deleting $n-r$ rows and the corresponding columns of $A$.  The next lemma is sometimes called the the Cauchy Interlacing Theorem, or the inclusion principle.

\begin{lem}[\cite{MR1084815}, Theorem 4.3.15]\label{interlacinglemma} 
Let $A\in\mathbb{C}^{n\times n}$ be Hermitian and $r\in \{1,\ldots,n\}$.  Then for all $k\in \{1,\ldots,r\}$,
\begin{center} $\lambda_{k+n-r}(A) \leq \lambda_{k}(A_r) \leq \lambda_{k}(A)$.
\end{center}
\end{lem}

The \emph{spectral radius} of a matrix $B\in \mathbb{C}^{n\times n}$ is $\rho(B):=\max\{|\lambda_i| : \lambda_i \text{ is an eigenvalue of } B \}$.  For $n\geq 2$, a matrix $B\in \mathbb{C}^{n\times n}$ is \emph{reducible} if there is a permutation matrix $P\in\mathbb{C}^{n\times n}$ such that $P^{\text{T}}BP$ is block upper triangular.  A matrix $B\in \mathbb{C}^{n\times n}$ is \emph{irreducible} if it is not reducible.  The following is part of the Perron-Frobenius Theorem.

\begin{lem}[\cite{MR1084815}, Theorem 8.4.4]\label{PFThm}  Suppose $A=(a_{ij})\in \mathbb{R}^{n\times n}$ is irreducible and $a_{ij}\geq 0$ for all $i,j\in \{1,\ldots,n\}$.  Then
\begin{enumerate}
\item $\rho(A)$ is an eigenvalue of $A$ and $\rho(A)>0$,
\item There exists an eigenvector $\mathbf{x}=(x_1,\ldots,x_n)\in \mathbb{R}^n$ with $x_i>0$ for each $i\in \{1,\ldots,n\}$ such that $A\mathbf{x}=\rho(A)\mathbf{x}$.
\end{enumerate} 
\end{lem}

\section{Incidence Matrices}
Suppose $\Phi=(\Gamma,\phi)$ is a $\mathbb{T}$-gain graph.  An \emph{incidence matrix} $\Eta(\Phi)=(\eta_{ve})$ is any $n\times m$ matrix, with entries in $\mathbb{T}\cup\{0\}$, where
\begin{equation*}
\eta_{v_i e}=
\begin{cases} -\eta_{v_j e}\phi(e_{ij}) &\text{if } e=e_{ij} \in E,
\\
0 &\text{otherwise;}
\end{cases}
\end{equation*}
furthermore, $\eta_{v_i e}\in \mathbb{T}$ if $e_{ij}\in E$.  We say ``an" incidence matrix, because with this definition $\Eta(\Phi)$ is not unique.  Each column can be multiplied by any element in $\mathbb{T}$ and the result can still be called an incidence matrix.  For example, we can choose $\eta_{v_j e}=1$ so $\eta_{v_ie}=-\phi(e_{ij})$ for each $e_{ij}\in E$.  This particular incidence matrix can be viewed as a generalization of an oriented incidence matrix of an unsigned graph.  Henceforth, we will write $\Eta(\Phi)$ to indicate that some fixed incidence matrix has been chosen.

\begin{lem}\label{LaplacianEtaEq} Let $\Phi=(\Gamma,\phi)$ be a $\mathbb{T}$-gain graph.  Then $L(\Phi)=\Eta(\Phi)\Eta(\Phi)^*$.
\end{lem}
\begin{proof}  The $(i,j)$-entry of $\Eta(\Phi)\Eta(\Phi)^*$ corresponds to the $i^{\text{th}}$ row of $\Eta(\Phi)$, indexed by $v_i \in V$, multiplied by the $j^{\text{th}}$ column of $\Eta(\Phi)^*$, indexed by $v_j \in V$.  Therefore, this entry is precisely $\sum_{e\in E} \eta_{v_i e}\bar{\eta}_{v_j e}$.  If $i=j$, then the sum simplifies to $\sum_{e\in E} |\eta_{v_i e}|^2=d_i$, since $|\eta_{v_i e}|=1$ if $v_i$ is incident to $e$.  If $i \neq j$, then since $\Gamma$ is simple, the sum simplifies to $\eta_{v_ie_{ij}}\bar{\eta}_{v_je_{ij}}= -\eta_{v_je_{ij}}\phi(e_{ij})\bar{\eta}_{v_je_{ij}}=- |\eta_{v_je_{ij}}|^2\phi(e_{ij})=- \phi(e_{ij})=-a_{ij}$.  Therefore, $\Eta(\Phi)\Eta(\Phi)^*=D(\Gamma)-A(\Phi)=L(\Phi)$. 
\end{proof}

Since $L(\Phi)=\Eta(\Phi)\Eta(\Phi)^*$, $L(\Phi)$ is positive semidefinite, and therefore, its eigenvalues are nonnegative.

Switching a $\mathbb{T}$-gain graph can be described as matrix multiplication of the incidence matrix, and matrix conjugation of the adjacency and Laplacian matrices.  For a switching function $\zeta$, we define a diagonal matrix $D(\zeta):=\text{diag}(\zeta(v_i):v_i\in V)$.  The following lemma shows how to calculate the switched gain graph's incidence, adjacency and Laplacian matrices.  

\begin{lem}\label{HALSwitched} Let $\Phi=(\Gamma,\phi)$ be a $\mathbb{T}$-gain graph.  Let $\zeta$ be a switching function on $\Phi$. Then
\begin{itemize}
\item $\Eta(\Phi^{\zeta}) = D(\zeta)^*\Eta(\Phi)$,
\item $A(\Phi^{\zeta})=D(\zeta)^*A(\Phi)D(\zeta)$,
\item $L(\Phi^{\zeta})=D(\zeta)^*L(\Phi)D(\zeta)$.
\end{itemize}
\end{lem}

\begin{lem}[\cite{MR2017726}, Theorem 2.1]\label{rankHThm} Let $\Phi=(\Gamma,\phi)$ be a $\mathbb{T}$-gain graph.  Then
\[ \rk(\Eta(\Phi))=n-b(\Phi). \]
\end{lem}

Since the ranks of $\Eta(\Phi)$ and $\Eta(\Phi)\Eta(\Phi)^*$ are the same the following is immediate.

\begin{cor}\label{rankLThm} Let $\Phi=(\Gamma,\phi)$ be a $\mathbb{T}$-gain graph.  Then 
\[ \rk(L(\Phi))=n-b(\Phi). \]
\end{cor}

\section{Eigenvalues of the Adjacency Matrix}
The next lemma implies that a switching class has an adjacency spectrum.  This is immediate from Lemma \ref{HALSwitched}.

\begin{lem}\label{simASwitch} Let $\Phi_1=(\Gamma,\phi_1)$ and $\Phi_2=(\Gamma,\phi_2)$ both be $\mathbb{T}$-gain graphs. If $\Phi_1\sim \Phi_2$, then $A(\Phi_1)$ and $A(\Phi_2)$ have the same spectrum.
\end{lem}

\begin{lem}\label{BalancedSpectrumA} If $\Phi=(\Gamma,\phi)$ is a balanced $\mathbb{T}$-gain graph, then $A(\Phi)$ and $A(\Gamma)$ have the same spectrum.
\end{lem}
\begin{proof}  By Lemma \ref{HBGen}, $\Phi\sim(\Gamma,1)$.  Thus, the result follows from Lemma \ref{simASwitch}. 
\end{proof}

The following is a generalization of a result for unsigned graphs \cite[p.2]{MR2571608}.

\begin{thm}\label{AdjSRub1} Let $\Phi=(\Gamma,\phi)$ be a $\mathbb{T}$-gain graph.  Then $\rho(A(\Phi)) \leq \Delta$.
\end{thm}
\begin{proof}
This proof is only a slight modification of the version for an unsigned graph.  Suppose $A(\Phi)\mathbf{x}=\lambda \mathbf{x}$ with $\mathbf{x}=(x_1,\ldots,x_n)\in \mathbb{C}^n\backslash\{\mathbf{0}\}$.  Then 
\[ \lambda x_i = \sum_{e_{ij} \in E(\Phi)} \phi(e_{ij}) x_j, \text{ for } i\in\{1,\ldots,n\}. \]
Let $|x_m| = \max _{i} |x_i|$.  Then
\[ |\lambda| |x_m| \leq \sum_{e_{mj} \in E(\Phi)} |\phi(e_{mj})| |x_j| \leq \Delta |x_m|.\]
Since $|x_m|\neq 0$ the proof is complete.  
\end{proof}

The following is a set of bounds which depend on the edge gains.  This is particularly interesting since the inequalities are not solely determined by the underlying graph of $\Phi$.  As lower bounds for $\lambda_1(A(\Phi))$, inequality \eqref{NeqAB1gain} is a generalization of a lower bound for the largest adjacency eigenvalue of an unsigned graph attributed to Collatz and Sinogowitz \cite{MR0087952}, and inequality \eqref{NeqAB2gain} is a generalization of a lower bound for the largest adjacency eigenvalue of an unsigned graph attributed to Hoffman \cite{MR0140441}.  

\begin{thm}\label{AdjacencyBounds1} Let $\Phi=(\Gamma,\phi)$ be a connected $\mathbb{T}$-gain graph.  Then
\begin{equation}\label{NeqAB1gain}
\lambda_{n}(A(\Phi)) \leq \frac{1}{n} \sum_{j=1}^n d_j^{\text{net}}  \leq \lambda_{1}(A(\Phi)),
\end{equation}
\begin{equation}\label{NeqAB2gain}
\lambda_{n}(A(\Phi)) \leq \sqrt{\frac{1}{n}\sum_{j=1}^n (d_j^{\text{net}})^2} \leq \lambda_{1}(A(\Phi)),
\end{equation}
and
\begin{equation}\label{NeqAB3gain}
\lambda_{n}(A(\Phi)) \leq \sqrt[3]{\frac{2}{n} \sum_{e_{ij}\in E(\Phi)} \text{Re}[\phi(e_{ij})]d_i^{\text{net}}d_j^{\text{net}}} \leq \lambda_{1}(A(\Phi)).
\end{equation}
\end{thm}

\begin{proof}  The proof method is inspired by \cite[Theorem 3.2.1]{MR2571608}.
For brevity, we write $A$ for $A(\Phi)$. Let $M_k={\bf j}^{\text{T}} A^k {\bf j}$.  From Lemma \ref{RRThm} the following is clear: 
\[ (\lambda_{n}(A))^k \leq M_k/\mathbf{j}^{\text{T}}\mathbf{j} \leq (\lambda_{1}(A))^k.  \]
We will compute $M_1$, $M_2$ and $M_3$; thus, making inequalities \eqref{NeqAB1gain}, \eqref{NeqAB2gain} and \eqref{NeqAB3gain} true.
We will use the equation:
\begin{equation}\label{netdegeq}
A{\bf j}= \big(\sum_{j=1}^n \phi(e_{1j}),\ldots,\sum_{j=1}^n \phi(e_{nj})\big)= \big(\sum_{\mathfrak{g}\in\mathcal{W}} \mathfrak{g} d_1^{\mathfrak{g}},\ldots,\sum_{\mathfrak{g}\in\mathcal{W}} \mathfrak{g} d_n^{\mathfrak{g}}\big)=(d_1^{\text{net}},\ldots, d_n^{\text{net}}) = \mathbf{d}^{\text{net}}.
\end{equation}

Now we compute $M_1$, $M_2$ and $M_3$.
\begin{align*}
M_1 &= {\bf j}^{\text{T}} A {\bf j}={\bf j}^{\text{T}}\mathbf{d}^{\text{net}} = \sum_{j=1}^n d_j^{\text{net}},\\
M_2 &= {\bf j}^{\text{T}} A^2 {\bf j}=(\mathbf{d}^{\text{net}})^{\text{T}}\mathbf{d}^{\text{net}} = \sum_{j=1}^n (d_j^{\text{net}})^2,\\
M_3 &= {\bf j}^{\text{T}} A^3 {\bf j}=(\mathbf{d}^{\text{net}})^{\text{T}} A \mathbf{d}^{\text{net}}=\sum_{i=1}^n (d_i^{\text{net}})\sum_{j=1}^n \phi(e_{ij})(d_j^{\text{net}})=2 \sum_{e_{ij}\in E(\Phi)} \text{Re}[\phi(e_{ij})]d_i^{\text{net}}d_j^{\text{net}}.\qedhere
\end{align*}
\end{proof}
\section{Eigenvalues of the Laplacian Matrix}
The next lemma implies that a switching class has a Laplacian spectrum.  The following is immediate from Lemma \ref{HALSwitched}.

\begin{lem}\label{simLSwitch} Let $\Phi_1=(\Gamma,\phi_1)$ and $\Phi_2=(\Gamma,\phi_2)$ both be $\mathbb{T}$-gain graphs. If $\Phi_1\sim \Phi_2$, then $L(\Phi_1)$ and $L(\Phi_2)$ have the same spectrum.
\end{lem}

\begin{lem}\label{BalancedSpectrumL} If $\Phi=(\Gamma,\phi)$ is a balanced $\mathbb{T}$-gain graph, then $L(\Phi)$ and $L(\Gamma)$ have the same spectrum.
\end{lem}
\begin{proof}  By Lemma \ref{HBGen}, $\Phi\sim(\Gamma,1)$.  Thus, the result follows from Lemma \ref{simLSwitch}. 
\end{proof}

The following is a simplification of the quadratic form $\mathbf{x}^{*} L(\Phi) \mathbf{x}$ which will be used to obtain eigenvalue bounds.

\begin{lem}\label{GGQuadraticForm}  Let $\Phi = (\Gamma,\phi)$ be a $\mathbb{T}$-gain graph. Suppose $\mathbf{x}=(x_1,x_2,\ldots,x_n)\in \mathbb{C}^n$.  Then
\[ \mathbf{x}^{*} L(\Phi) \mathbf{x} = \sum_{e_{ij}\in E(\Phi)} |x_i - \phi(e_{ij}) x_j |^2.\]
\end{lem}

The following theorem establishes a relationship between the Laplacian spectral radius of a $\mathbb{T}$-gain graph and the signless Laplacian spectral radius of an unsigned graph.  This generalizes a signed graphic version that appears in \cite{MR1950410}, which generalizes the unsigned graphic relation: $\lambda_1(L(\Gamma))\leq\lambda_1(Q(\Gamma))$.  

\begin{thm}\label{LapSLCompThm}
Let $\Phi = (\Gamma,\phi)$ be a connected $\mathbb{T}$-gain graph.  Then 
\[\lambda_{1}(L(\Phi))\leq \lambda_{1}(L(\Gamma,-1))=\lambda_1(Q(\Gamma)).\]
Furthermore, equality holds if and only if $\Phi \sim (\Gamma,-1)$.
\end{thm}
\begin{proof}
The proof is similar to the signed graphic proof in \cite[Lemma 3.1]{MR1950410}.  

Let $\mathbf{x}=(x_1,x_2,\ldots,x_n)^{\text{T}}\in \mathbb{C}^n$ be a unit eigenvector of $L(\Phi)$ with corresponding eigenvalue $\lambda_{1}(L(\Phi))$.  By Lemma \ref{GGQuadraticForm}:

\[ \lambda_{1}(L(\Phi)) = \mathbf{x}^* L(\Phi) \mathbf{x} =  \sum_{e_{ij}\in E(\Phi)} |x_i - \phi(e_{ij}) x_j |^2
\leq \sum_{e_{ij}\in E(\Phi)} (|x_i| + |x_j|)^2 \leq \max_{y^*y=1} \sum_{e_{ij}\in E(\Phi)} (|y_i| + |y_j|)^2. \]
Since $L(\Gamma,-1)$ is nonnegative, and $\Gamma$ is connected, $L(\Gamma,-1)$ is irreducible.  Hence, by Lemma \ref{PFThm}, there is an eigenvector $\mathbf{y}=(y_1,\ldots,y_n)\in \mathbb{R}^n$ of $L(\Gamma,-1)$, with corresponding eigenvalue $\lambda_1(L(\Gamma,-1))$, where $y_i>0$ for every $i\in \{1,\ldots,n\}$.  Therefore, $|y_k|=y_k$ for all $k\in \{1,\ldots,n \}$. Finally, by Lemma \ref{RRThm}, $\displaystyle\max_{y^*y=1} \sum_{e_{ij}\in E(\Phi)} (|y_i| + |y_j|)^2=\lambda_{1}(L(\Gamma,-1))$.

If $\Phi \sim (\Gamma,-1)$, then $\lambda_{1}(L(\Phi))=\lambda_{1}(L(\Gamma,-1))$ by Lemma \ref{simLSwitch}.

Now suppose that $\lambda_1(L(\Gamma,-1))=\lambda_1(L(\Phi))$.  Then $|x_i|^2+2|x_i||x_j|+|x_j|^2 = |x_i - \phi(e_{ij}) x_j |^2= |x_i|^2 - 2\cdot \text{Re}(\bar{x}_i x_j \phi(e_{ij}))+ |x_j|^2$. Therefore, $- \text{Re}(\bar{x}_i x_j \phi(e_{ij})) = |x_i||x_j|$.  Also, $(x_1,\ldots,x_n)$ is an eigenvector of $L(\Gamma,-1)$ with corresponding eigenvalue $\lambda_1(L(\Gamma,-1))$.  From Lemma \ref{PFThm}, $x_l\neq0$ for every $l\in\{1,\ldots,n\}$.  Let $x_i=|x_i|e^{i \theta_i}$, $x_j=|x_j|e^{i \theta_j}$ and $\phi(e_{ij})=e^{i\theta}$.  Thus, $- \text{Re}(|x_i||x_j|e^{-i \theta_i}e^{i \theta_j} e^{i\theta}) = |x_i||x_j|$.  That is, $\text{Re}(e^{i(-\theta_i+ \theta_j+\theta)}) = -1$.  So $e^{i(-\theta_i+ \theta_j+\theta)} = -1$.  Finally, by substitution, $\frac{\bar{x}_i x_j}{|x_i||x_j|}\phi(e_{ij})=-1$; that is, 
$\frac{x_i \bar{x}_j}{|x_i||x_j|}=-\phi(e_{ij})$.  

Let $\theta:V\rightarrow \mathbb{T}$ be defined as $\theta(v_i)=\bar{x}_i/|x_i|$, for all $v_i\in V$.  Since $\bar{x}_i/|x_i|\in \mathbb{T}$ for every $i\in\{1,\ldots,n\}$, $\theta$ is a potential function for $-\overline{\phi}$.  Therefore, $(\Gamma,-\overline{\phi})$ is balanced by Lemma \ref{HBGen}.
\end{proof}

Theorem \ref{LapSLCompThm} says that any known upper bound for the signless Laplacian spectral radius of an unsigned graph is also an upper bound for the Laplacian spectral radius of a $\mathbb{T}$-gain graph.  Moreover, equality holds if and only if the gain graph is switching equivalent to $(\Gamma,-1)$.  Consequently, we can state the following corollary, which includes many known upper bounds for the signless Laplacian of an unsigned graph.  The organization of the following bounds is inspired by \cite{MR2588119}.  References are provided for each bound, indicating where these upper bounds are known for the signless Laplacian spectral radius.

\begin{cor} Let $\Phi=(\Gamma,\phi)$ be a connected $\mathbb{T}$-gain graph.  Then the following upper bounds on the Laplacian spectral radius are valid:

\noindent (1) Upper bounds depending on $d_i$ and $m_i$:
\begin{align*}
\lambda_1(L(\Phi))&\leq 2\Delta,\text{\qquad\cite{MR2312332}}\\
\lambda_1(L(\Phi))&\leq \max_i\{d_i+m_i\},\text{\qquad\cite{MR2015532}}\\
\lambda_1(L(\Phi))&\leq \max_i\{d_i+\sqrt{d_im_i}\},\text{\qquad\cite{MR2588119}}\\
\lambda_1(L(\Phi))&\leq \max_i\{\sqrt{2d_i(d_i+m_i)}\}, \text{\qquad\cite{MR2588119}}\\
\lambda_1(L(\Phi))&\leq \max_i\left\{\frac{d_i+\sqrt{(d_i)^2+8(d_im_i)}}{2}\right\}. \text{\qquad\cite{MR2588119}}
\end{align*}

\noindent (2) Upper bounds depending on $d_i$, $d_j$, $m_i$ and $m_j$:
\begin{align*}
\lambda_1(L(\Phi)) &\leq \max_{e_{ij}\in E}\{d_i+d_j \},\text{\qquad\cite{MR2312332}}\\
\lambda_1(L(\Phi)) &\leq \max_{e_{ij}\in E}\left\{\frac{d_i(d_i+m_i)+d_j(d_j+m_j)}{d_i+d_j} \right\},\text{\qquad\cite{MR1653547, MR2588119}}\\
\lambda_1(L(\Phi)) &\leq \max_{e_{ij}\in E}\left\{\sqrt{d_i(d_i+m_i)+d_j(d_j+m_j)}\right\}, \text{\qquad\cite{MR2015534, MR2588119}}\\
\lambda_1(L(\Phi)) &\leq \max_{e_{ij}\in E}\left\{2+\sqrt{d_i(d_i+m_i-4)+d_j(d_j+m_j-4)+4} \right\}, \text{\qquad\cite{MR2015534, MR2588119}}\\
\lambda_1(L(\Phi)) &\leq \max_{e_{ij}\in E}\left\{\frac{d_i+d_j+\sqrt{(d_i-d_j)^2+4m_im_j}}{2} \right\}. \text{\qquad\cite{MR2031534}}
\end{align*}
Furthermore, equality holds if and only if $\Phi \sim (\Gamma,-1)$.
\end{cor}


\begin{figure}[h!]
    \includegraphics[width=0.5\textwidth]{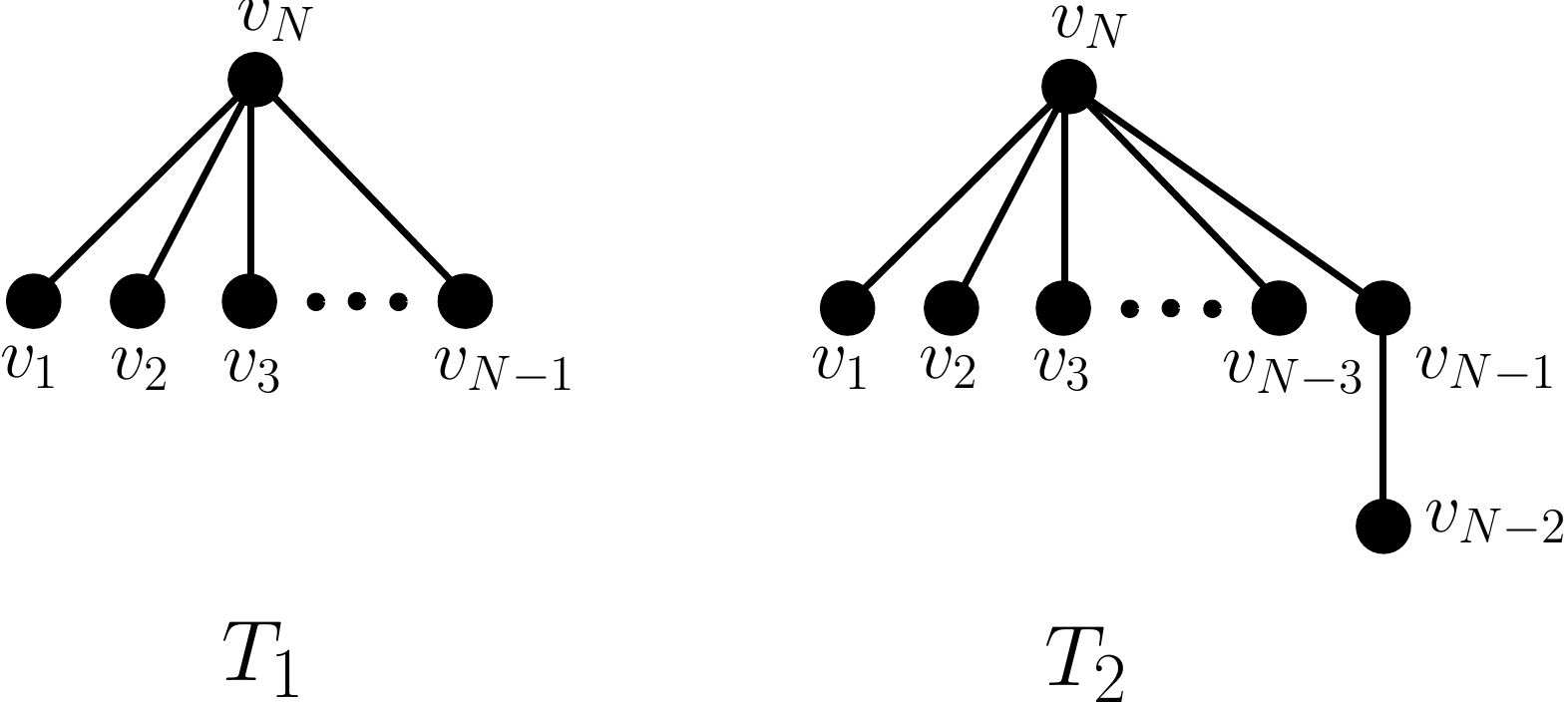}
\caption{Trees considered in Lemma \ref{tree12lapsr}.}\label{treelaplemma}
\end{figure}

\begin{lem}\label{tree12lapsr}
Consider the trees $T_1$ (with $N\geq 3$) and $T_2$ (with $N\geq 4$) in Figure \ref{treelaplemma}.  The following is true:
\begin{align*}
\lambda_1(L(T_1)) &= \Delta+1,\\
\lambda_1(L(T_2)) &> \Delta+1.
\end{align*}
\end{lem}
\begin{proof}
The characteristic polynomial of $L(T_1)$ is 

\begin{align*}
p_{L(T_1)}(\lambda)&=\left|\begin{matrix} \lambda-1 & 0 & 0  & \cdots &0 & 1 \\ 0 & \lambda-1 & 0 & \cdots &0 & 1  \\ 0 & 0 & \lambda-1  & \ddots &\vdots & \vdots \\ \vdots & \vdots &  \ddots  & \ddots & 0 & 1 \\ 0 & 0 &  0  & \cdots & \lambda-1 & 1 \\ 1 & 1 &  1  & \cdots & 1 & \lambda-(N-1) \\ \end{matrix}\right|
\\
&=\left|\begin{matrix} \lambda-1 & 0 & 0  & \cdots &0 & 1 \\ 0 & \lambda-1 & 0 & \cdots &0 & 1  \\ 0 & 0 & \lambda-1  & \ddots &\vdots & \vdots \\ \vdots & \vdots &  \ddots  & \ddots & 0 & 1 \\ 0 & 0 &  0  & \cdots & \lambda-1 & 1 \\ 0 & 0 &  0  & \cdots & 0 & \lambda-(N-1)-\frac{N-1}{\lambda-1} \\ \end{matrix}\right|\\
&= \lambda(\lambda-1)^{N-2}(\lambda-N).
\end{align*}
Therefore, $\lambda_1(L(T_1))=N=\Delta+1$.  

The characteristic polynomial of $L(T_2)$ is:
\begin{align*}
p_{L(T_2)}(\lambda)
&=\left|\begin{matrix} \lambda-1 & 0 & \cdots &0 &0 &0 & 1 \\ 0 & \ddots & \ddots & \ddots &0&0 & 1  \\ \vdots & \ddots & \ddots & \ddots &\vdots  & \vdots\\ 0 &  \cdots&0 &  \lambda-1 & 0 & 0 & 1 \\ 0 & \cdots&0 &0&  \lambda-1 & 1 & 0 \\ 0 &  \cdots&0 &0& 1 & \lambda-2 & 1 \\ 1  & \cdots&1 &1&  0 & 1 & \lambda-(N-2) \\ \end{matrix}\right|\\
&=\left|\begin{matrix} \lambda-1 & 0 & \cdots &0 &0 &0 & 1 \\ 0 & \ddots & \ddots & \ddots &0&0 & 1  \\ \vdots & \ddots & \ddots & \ddots &\vdots  & \vdots\\ 0 &  \cdots&0 &  \lambda-1 & 0 & 0 & 1 \\ 0 & \cdots&0 &0&  \lambda-1 & 1 & 0 \\ 0 &  \cdots&0 &0& 0 & \lambda-2-\frac{1}{\lambda-1} & 1 \\ 0  & \cdots&0 &0&  0 & 0 & \lambda-(N-2)-\frac{N-3}{\lambda-1}-\frac{1}{\lambda-2-\frac{1}{\lambda-1}} \\ \end{matrix}\right|\\
&=\lambda(\lambda-1)^{N-4}(\lambda^3-(N+2)\lambda^2-(2-3N)\lambda-N).
\end{align*}
Notice that $p_{L(T_2)}(\Delta+1) <0$.  Therefore, $p_{L(T_2)}(\lambda)$ has a root larger than $\Delta+1$.
\end{proof}
Let $\Phi\backslash e$ be the gain graph obtained from $\Phi=(\Gamma,\phi)$ by deleting the edge $e$ from $\Gamma$.  The following lemma establishes a relationship between the eigenvalues of $L(\Phi)$ and $L(\Phi\backslash e)$.  This generalizes the signed and unsigned graphic versions found in \cite{MR1950410} and \cite[p. 187]{MR2571608}, respectively.  

\begin{lem}\label{gainLinterlaceEdges}  Let $\Phi = (\Gamma, \phi)$ be a $\mathbb{T}$-gain graph.  Then
\[ \lambda_{k+1}(L(\Phi)) \leq \lambda_{k}(L(\Phi\backslash e)) \leq \lambda_{k}(L(\Phi)), \text{ for all }k\in \{1,\ldots, n\}. \]
\end{lem}
\begin{proof}
The proof is identical to that of \cite[Lemma 3.7]{MR1950410}.
Notice that $\Eta(\Phi\backslash e)^{*}\Eta(\Phi\backslash e)$ is a principal submatrix of $\Eta(\Phi)^{*}\Eta(\Phi)$.  
Since the nonzero eigenvalues of $\Eta(\Phi)\Eta(\Phi)^{*}$ and $\Eta(\Phi)^{*}\Eta(\Phi)$ are the same, the result follows from Lemma \ref{interlacinglemma}.
\end{proof}

The following is a lower bound on the Laplacian spectral radius of a $\mathbb{T}$-gain graph, which depends only on the maximum degree.  This generalizes a signed graphic bound that appears in \cite{MR1950410}, which generalizes an unsigned graphic version in \cite[p.186]{MR2571608}.
\begin{thm}  Let $\Phi=(\Gamma,\phi)$ be a $\mathbb{T}$-gain graph.  Then
\[ \Delta +1 \leq \lambda_1(L(\Phi)). \]
Furthermore, when $\Gamma$ is connected, equality holds if and only if $\Delta=n-1$ and $\Phi$ is balanced. 
\end{thm}
\begin{proof}
The proof uses similar techniques as in \cite[Theorem 3.9]{MR1950410} and \cite[Theorem 3.10]{MR1950410}.  

There exists a subtree $T_1$ (see Figure \ref{treelaplemma}) of $\Phi$ with $N=\Delta+1$.  Since $T_1$ is a tree it is balanced.  From Lemmas \ref{tree12lapsr} and \ref{BalancedSpectrumL}, $\lambda_1(L(T_1))=\Delta+1$.  By repeated application of Lemma \ref{gainLinterlaceEdges}: $\lambda_1(L(\Phi)) \geq \lambda_1(L(T_1)) = \Delta+1.$

Suppose that $\Gamma$ is connected, $\Phi$ is balanced and $\Delta=n-1$.  Since $\Phi$ is balanced $\Phi \sim (\Gamma,1)$ by Lemma \ref{HBGen}, which means that $\lambda_1(L(\Phi))=\lambda_1(L(\Gamma,1))$ by Lemma \ref{simLSwitch}.  For an unsigned graph the result is well known, see \cite[p.186]{MR2571608} for example.  Therefore, $\Delta+1=\lambda_1(L(\Phi))$.

Suppose that $\Gamma$ is connected and $\Delta+1=\lambda_1(L(\Phi))$.  If $\Delta<n-1$, then there exists a subtree $T_2$ (see Figure \ref{treelaplemma}) of $\Phi$.  Since $T_2$ is a tree it is balanced.  From Lemmas \ref{tree12lapsr} and \ref{BalancedSpectrumL}, $\lambda_1(L(T_2))>\Delta+1$.  By repeated application of Lemma \ref{gainLinterlaceEdges}: $\lambda_1(L(\Phi))\geq\lambda_1(L(T_2))>\Delta+1$.  This contradicts the hypothesis that $\Delta+1=\lambda_1(L(\Phi))$.  Therefore, we may further assume that $\Delta=n-1$.

Let $v$ be a vertex with $\text{deg}(v)=n-1$.  Hence, if $\Phi$ has a cycle there must be a subgraph consisting of a single 3-cycle containing $v$ and $n-3$ pendant edges (an edge where one of its vertices has degree 1) incident to $v$.  Call this subgraph $G$.

We show that $G$ must be balanced.  The structure $G$ allows us to switch all edges incident to $v$ to be neutral.  The switched graph is denoted by $\hat{G}$. By Lemma \ref{simLSwitch}, the $\hat{G}$ and $G$ have the same characteristic polynomial.  Suppose the gain of the 3-cycle belonging to $G$ is $\mathfrak{g} \in \mathbb{T}$.  The characteristic polynomial of $L(G)$ is 
\[ p_{L(G)}(\lambda)=p_{L(\hat{G})}(\lambda)=(\lambda-1)^{n-3}[\lambda^3+\lambda^2(-n-3)+3n\lambda+(2\text{Re}(\mathfrak{g})-2)].\]
The characteristic polynomial evaluated at $\Delta+1=n$ is $p_{L(G)}(n)=(n-1)^{n-3}(2\text{Re}(\mathfrak{g})-2)$.  If $n\geq3$ and $\mathfrak{g} \neq 1$, then $p_{L(G)}(n)<0$.  Hence, $p_{L(G)}(\lambda)$ has a root larger than $\Delta+1$.  If $\mathfrak{g} \neq 1$, then by repeated application of Lemma \ref{gainLinterlaceEdges}: $\lambda_1(L(\Phi))\geq\lambda_1(L(G))>\Delta+1$.  This contradicts the hypothesis that $\Delta+1=\lambda_1(L(\Phi))$.  Therefore, we may further assume that $\mathfrak{g}=1$; that is, $G$ is balanced.  This means any triangle containing $v$ is balanced.

Every cycle in $\Phi$ is the symmetric difference of triangles containing $v$.  Therefore, every cycle in $\Phi$ is balanced.  For a proof of this inference of balance see \cite[Corollary 3.2]{MR1007712}.  
\end{proof}
Here we present a set of Laplacian eigenvalue bounds which actually depend on the gain function. 

\begin{thm}\label{thm1} Let $\Phi=(\Gamma,\phi)$ be a connected $\mathbb{T}$-gain graph.  Then
\begin{equation}\label{NeqLB1gain}
\lambda_{n}(L(\Phi)) \leq \frac{1}{n} \sum_{j=1}^n (d_j- d_j^{\text{net}})  \leq \lambda_{1}(L(\Phi)),
\end{equation}
\begin{equation}\label{NeqLB2gain}
\lambda_{n}(L(\Phi)) \leq \sqrt{\frac{1}{n}\sum_{j=1}^n (d_j- d_j^{\text{net}})^2} \leq \lambda_{1}(L(\Phi)),
\end{equation}
and
\begin{equation}\label{NeqLB3gain}
\lambda_{n}(L(\Phi)) \leq \sqrt[3]{\frac{1}{n}\sum_{j=1}^n d_j(d_j-d_j^{\text{net}})^2 - \frac{2}{n}\sum_{e_{ij}\in E} \text{Re}[\phi(e_{ij})](d_i-d_i^{\text{net}})(d_j-d_j^{\text{net}})} \leq \lambda_{1}(L(\Phi)).
\end{equation}
\end{thm}

\begin{proof}
For brevity, we write $D$ for $D(\Phi)$ and $A$ for $A(\Phi)$.  Let $N_k={\bf j}^{\text{T}} L(\Phi)^k {\bf j}$.  From Lemma \ref{RRThm} the following is clear: 
\[ (\lambda_{n}(L(\Phi)))^k \leq N_k/\mathbf{j}^{\text{T}}\mathbf{j} \leq (\lambda_{1}(L(\Phi)))^k.  \]
We will compute $N_1$, $N_2$ and $N_3$; thus, making inequalities \eqref{NeqLB1gain}, \eqref{NeqLB2gain} and \eqref{NeqLB3gain} true. 
Now we compute $N_1$.
\begin{align*}
N_1 &= {\bf j}^{\text{T}} L(\Phi) {\bf j}={\bf j}^{\text{T}} (D-A) {\bf j} = {\bf j}^{\text{T}}(\mathbf{d}-\mathbf{d}^{\text{net}}) = \sum_{j=1}^n (d_j- d_j^{\text{net}}).
\end{align*}

Similarly we compute $N_2$. 
\begin{align*}
N_2 &= {\bf j}^{\text{T}} L(\Phi)^2 {\bf j}= {\bf j}^{\text{T}} (D^2-AD-DA+A^2) {\bf j} \\
&= {\bf j}^{\text{T}}(D\mathbf{d}-A\mathbf{d}-D\mathbf{d}^{\text{net}}+A\mathbf{d}^{\text{net}} \\
&= \mathbf{d}^{\text{T}}\mathbf{d}-(\mathbf{d}^{\text{net}})^{\text{T}}\mathbf{d}-\mathbf{d}^{\text{T}}\mathbf{d}^{\text{net}}+(\mathbf{d}^{\text{net}})^{\text{T}}\mathbf{d}^{\text{net}} \\
&=\sum_{j=1}^n d_j^2-2\sum_{j=1}^n d_j^{\text{net}}d_j + \sum_{j=1}^n (d_j^{\text{net}})^2.
\end{align*}
To compute $N_3$ we will consider the matrices formed in the expansion of $L(\Phi)^3$.  Using equation \eqref{netdegeq} it is easy to check that
\begin{equation*}
{\bf j}^{\text{T}} DAD {\bf j}= \mathbf{d}^{\text{T}}A\mathbf{d} = \sum_{i=1}^n d_i \sum_{j=1}^n \phi(e_{ij}) d_j = 2\sum_{e_{ij}\in E(\Phi)} \text{Re}[\phi(e_{ij})]d_i d_j,
\end{equation*}
\begin{equation*}
{\bf j}^{\text{T}} AD^2 {\bf j}= (\mathbf{d}^{\text{net}})^{\text{T}} D\mathbf{d} =(\mathbf{d}^{\text{net}})^{\text{T}} \mathbf{d}^{(2)} =\sum_{j=1}^n (d_j^{\text{net}})d_{j}^2,
\end{equation*}
\begin{equation*}
{\bf j}^{\text{T}} DA^2 {\bf j}=\mathbf{d}^{\text{T}}A\mathbf{d}^{\text{net}} =\sum_{i=1}^n d_i \sum_{j=1}^n \phi(e_{ij}) (d_j^{\text{net}})=\sum_{e_{ij}\in E(\Phi)} [\phi(e_{ij})d_j^{\text{net}}d_i+\overline{\phi(e_{ij})}d_i^{\text{net}}d_j],
\end{equation*}
\begin{equation*}
{\bf j}^{\text{T}} ADA {\bf j}=(\mathbf{d}^{\text{net}})^{\text{T}} D\mathbf{d}^{\text{net}} = \sum_{j=1}^n (d_j^{\text{net}})^2d_{j},  
\end{equation*}
\begin{equation*}
{\bf j}^{\text{T}} D^2A {\bf j}=\mathbf{d}^{\text{T}} D\mathbf{d}^{\text{net}} =\mathbf{d}^{\text{T}} D\mathbf{d}^{\text{net}} =(\mathbf{d}^{(2)})^{\text{T}} \mathbf{d}^{\text{net}} = \sum_{j=1}^n (d_j^{\text{net}})d_{j}^2= {\bf j}^{\text{T}} AD^2 {\bf j},
\end{equation*}
and
\begin{equation*}
{\bf j}^{\text{T}} A^2D {\bf j}=(\mathbf{d}^{\text{net}})^{\text{T}} A \mathbf{d}
=\sum_{i=1}^n d_i^{\text{net}} \sum_{j=1}^n \phi(e_{ij}) d_j=\sum_{e_{ij}\in E(\Phi)} [\overline{\phi(e_{ij})}d_j^{\text{net}}d_i+\phi(e_{ij})d_i^{\text{net}}d_j].
\end{equation*}
The calculation for ${\bf j}^{\text{T}} A^3 {\bf j}$ is done in the proof of Theorem \ref{AdjacencyBounds1}.  Now the following simplification of $N_3$ can be made.
\begin{align*}
N_3 &= {\bf j}^{\text{T}} L(\Phi)^3 {\bf j}= {\bf j}^{\text{T}}(D^3-AD^2-DAD+A^2D-D^2A+ADA+DA^2-A^3){\bf j}  \\
&= \sum_{j=1}^n d_j^3-2\sum_{j=1}^n d_j^{\text{net}}d_{j}^2 - 2\sum_{e_{ij}\in E(\Phi)} \text{Re}[\phi(e_{ij})]d_i d_j + \sum_{j=1}^n (d_j^{\text{net}})^2d_{j} \\
& \qquad + 2\sum_{e_{ij}\in E(\Phi)} \text{Re}[\phi(e_{ij})](d_j^{\text{net}}d_i+d_i^{\text{net}}d_j)- 2\sum_{e_{ij}\in E(\Phi)} \text{Re}[\phi(e_{ij})]d_i^{\text{net}}d_j^{\text{net}}\\
&= \sum_{j=1} d_j(d_j-d_j^{\text{net}})^2 - 2\sum_{e_{ij}\in E} \text{Re}[\phi(e_{ij})](d_i-d_i^{\text{net}})(d_j-d_j^{\text{net}}).\qedhere
\end{align*}
\end{proof}
Since every edge has two (not necessarily distinct) group elements associated to it, we define a subgraph induced by a pair of group elements, $\{\mathfrak{g},\mathfrak{g}^{-1}\}$. The set of \emph{inverse pairs} is $\mathfrak{U}:=\{\{\mathfrak{g},\mathfrak{g}^{-1}\}:\mathfrak{g}\in\mathfrak{W} \}$.  Let $S:=\{ e_{ij}\in E : \phi(e_{ij})=\mathfrak{g} \text{ or } \phi(e_{ij})=\mathfrak{g}^{-1} \}$.  The subgraph induced by the pair $\{\mathfrak{g},\mathfrak{g}^{-1}\}\in \mathfrak{U}$ is the gain graph $\Phi^{\{\mathfrak{g},\mathfrak{g}^{-1}\}}:=(V,S,\mathbb{T},\phi{\mid}_S)$.
  
The next theorem is a generalization of the signed and and unsigned graphic results in \cite{MR1950410} and \cite[p.194]{MR2571608}, respectively.

\begin{thm} Let $\Phi = (\Gamma,\phi)$ be a connected $\mathbb{T}$-gain graph.  Then
\[ \max_{\{\mathfrak{g},\mathfrak{g}^{-1}\}\in \mathfrak{U}}\lambda_{1}(L(\Gamma^{\{\mathfrak{g},\mathfrak{g}^{-1}\}})) \leq \lambda_{1}(L(\Phi)) \leq \sum_{\{\mathfrak{g},\mathfrak{g}^{-1}\}\in \mathfrak{U}} \lambda_{1}(L(\Gamma^{\{\mathfrak{g},\mathfrak{g}^{-1}\}})).\]
\end{thm}
\begin{proof} The proof is inspired by \cite[Corollary 3.8]{MR1950410}.  The lower bound is immediate by repeated application of Corollary \ref{gainLinterlaceEdges}.  Also, notice that $L(\Phi)=\sum_{\{\mathfrak{g},\mathfrak{g}^{-1}\}\in \mathfrak{U}} L(\Gamma^{\{\mathfrak{g},\mathfrak{g}^{-1}\}})$.  Hence, from Lemma \ref{WeylINEQ} the right inequality is valid.
\end{proof}

\section{Examples - The cycle and path $\mathbb{T}$-gain graphs}

Here we calculate the adjacency and Laplacian eigenvalues of the cycle and path graphs with complex unit gains.  Let $C_n$ be a cycle on $n$ vertices.  The following calculation is a generalization of the unsigned and signed graph versions in \cite[p.3]{MR2571608} and \cite{Germina20112432}, respectively.
 
\begin{thm} Suppose $\Phi=(C_n,\mathbb{T},\phi)$ with $\phi(C_n)=\xi=e^{i\theta}$.  Then
\begin{equation}
\sigma(A(\Phi))=\left\{ 2\cos\left(\frac{\theta+2\pi j}{n}\right): j\in\{0,\ldots,n-1\} \right\},
\end{equation}
and
\begin{equation}
\sigma(L(\Phi))=\left\{ 2-2\cos\left(\frac{\theta+2\pi j}{n}\right) : j\in\{0,\ldots,n-1\} \right\}.
\end{equation}
\end{thm}
\begin{proof} This proof is only a slight modification of the version for an unsigned graph in \cite[p.3]{MR2571608}.  Switch the cycle so that at most one edge is nonneutral, say $\phi(v_nv_1)=\xi$. Let
\[
 P = \begin{bmatrix}
       0 & 1 & 0 &  \cdots & 0\\
       0 & 0 & 1 &  \cdots & 0\\
       \vdots & \vdots & \ddots & \ddots & \vdots \\
       0 & 0 & 0 &  \cdots & 1\\
       \xi & 0 & 0 & \cdots & 0\\
     \end{bmatrix}.
\]
It is obvious that $PP^*=I$, so $P^*=P^{-1}$.  Let $\mathbf{x}=(x_1,\ldots,x_n)\in\mathbb{C}^n$ be an eigenvector of $P$ with corresponding eigenvalue $\lambda(P)$.  From the equation $P\mathbf{x}=\lambda(P)\mathbf{x}$ it is clear that $x_i=\lambda(P)x_{i-1}$ for $i\in \{2,\ldots,n\}$, and $\xi x_1=\lambda(P)x_n$.  Therefore, $\lambda(P)^n-\xi=0$, and so $\sigma(P)=\{ \xi^{1/n}e^{2 \pi i j/n} : j\in\{0,1,\ldots,n-1\} \}$.  Notice that $A=P+P^*=P+P^{-1}$.  Thus, $A$ has eigenvalues of the form $\lambda(A)=\lambda(P)+\lambda(P)^{-1}=2\text{ Re}[\xi^{1/n}e^{2 \pi i j/n}]=2\cos(\frac{\theta+2\pi j}{n})$.  

Since $\Phi$ is 2-regular we can write $L(\Phi)=2I-A(\Phi)$.  Hence, if $\mathbf{x}$ is an eigenvector of $A(\Phi)$ with associated eigenvalue $\lambda$, then $L(\Phi)\mathbf{x}=(2I-A(\Phi))\mathbf{x}=(2-\lambda)\mathbf{x}$.  The result follows. 
\end{proof}

  Let $P_n$ be a path graph with $n$ vertices.  Since any $\mathbb{T}$-gain graph $\Phi=(P_n,\phi)$ is balanced, $\Phi$ and $P_n$ have the same spectrum by Lemma \ref{BalancedSpectrumL}.  The eigenvalues of $P_n$ can be found in \cite[p.47]{MR2571608}.

\section{ Acknowledgements}
The author would like to thank Thomas Zaslavsky for his valuable comments and suggestions regarding this work.


\bibliographystyle{amsplain}
\bibliography{mybib}

\end{document}